\documentclass[12pt,reqno]{article}
\usepackage{amssymb,amscd,amsmath,amsthm,color}
\newtheorem{theorem}{Theorem}[section]
\newtheorem{lemma}[theorem]{Lemma}
\newtheorem{cor}[theorem]{Corollary}
\newtheorem{prop}[theorem]{Proposition}

\usepackage{enumerate,amssymb}
\theoremstyle{definition}

\theoremstyle{remark}

\numberwithin{equation}{section}

\def\bC{\mathbb{C}}

\def\bM{\mathbb{M}}

\newcommand{\sym}{\mathrm{sym}}
\newcommand{\qsym}{\mathrm{qsym}}

\begin{document}
\baselineskip=15pt

\title{Some hybrid matrix triangle inequalities }

\author{ Jean-Christophe Bourin and Eun-Young Lee}

\date{ }

\maketitle

\vskip 10pt\noindent
{\small
{\bf Abstract.} A  recent result due to  Teng Zhang compares the sum of $m$ matrices and the sum of their  quadratic symmetric moduli:
$$
\left\| \sum_{k=1}^m A_k\right\| \le \sqrt{2}  \left\| \sum_{k=1}^m |A_k|_{\qsym}\right\|
$$
for every unitarily invariant norm. Here $|A|_{\qsym}$ is the quadratic mean of $|A|$ and $|A^*|$. We derive operator and eigenvalue refinements of Zhang's inequality from a new polar decomposition for the quadratic symmetric modulus.  For instance, 
$$
\left| \sum_{k=1}^m A_k\right| \le \frac{\sqrt{2}}{2} \left\{ \sum_{k=1}^m \left(|A_k|_{\qsym}+V|A_k|_{\qsym}V^*\right)\right\}
$$
for some unitary matrix $V$. We also establish the polar decomposition  for the maximal modulus  associated with Olson's order,  and  derive, as in the quadratic case, a series of estimates.
}

\section{Introduction}

In \cite{Zsharp}, Zhang initiated a systematic study of matrix triangle inequalities involving different operator moduli.  We refer to such estimates as hybrid triangle inequalities. This paper aims to refine one of Zhang's theorems, stated in the abstract,  and thereby contribute to the emerging theory of symmetrized moduli. While Zhang's theorem is formulated in terms of weak majorization, our main results provide operator inequalities involving unitary orbits. Consequently, they yield refined eigenvalue and norm estimates.

  We denote by $\bM_d$  the space of complex $d\times d$ matrices and $\bM_d^+$  its cone of positive semidefinite matrices. For $A\in\bM_d$, we write $|A|=(A^*A)^{1/2}$  for its absolute value (or right modulus).  If $A\in\bM_d^+$, we denote by 
$\lambda_1^{\downarrow}(A)\ge \cdots\ge  \lambda_d^{\downarrow}(A)$  its eigenvalues arranged  in   nonincreasing order. For  $A,B\in\bM_d^+$, the weak majorization, or submajorization, relation $A\prec_w B$ means that
 \begin{equation}\label{submajdef}
\sum_{j=1}^k\lambda_j^{\downarrow}(A) \le  \sum_{j=1}^k\lambda_j^{\downarrow}(B), \qquad k=1,\ldots, d.
\end{equation}
We   adhere to the convention that, for $j>d$ and $A\in\bM_d^+$, $\lambda_j^{\downarrow}(A):=0$. 
A basic property of symmetric norms (i.e., unitarily invariant ones) on $\bM_d$ is Fan's dominance principle:  if $A,B\in\bM^+_d$, then
$$
\| A\| \le \| B\| \quad {\mathrm{for \ all  \ symmetric \ norms}} \iff |A|\prec_w |B|.
$$
We refer to \cite{HiP} for a background and details on majorization, symmetric norms, and the matrix geometric mean.

For a matrix $Z\in\bM_d$, its left and right moduli $|Z^*|$ and $|Z|$ are the positive parts in the polar decompositions
$$
Z=|Z^*|U=|Z^*|^{1/2}U|Z|^{1/2}=U|Z|.
$$
These moduli occur in a number of important operator inequalities. Since neither modulus enjoys an intrinsic preference, it is natural to consider symmetrized versions obtained by averaging them.
The symmetric modulus and the quadratic symmetric modulus of $Z\in\bM_d$ are respectively defined as
$$
|Z|_{\sym}:= \frac{|Z|+|Z^*|}{2}, \qquad
|Z|_{\qsym}:=\sqrt{\frac{|Z|^2+|Z^*|^2}{2}}.
$$
Note that $$|Z|_{\qsym}=\sqrt{|{\mathrm{Re\,}}Z|^2+ |{\mathrm{Im\,}}Z|^2}$$
and
 \begin{equation}\label{compsqs}
|Z|_{\sym}\le |Z|_{\qsym}
\end{equation} 
by operator concavity of $\sqrt{t}$.  The quadratic symmetric modulus enjoys remarkably strong subadditivity properties. A striking result of Zhang \cite[Theorem 1.9]{Ztriang} states that:

\vskip 5pt\noindent
\begin{theorem}\label{thZ}  Let $A,B\in\bM_d$. Then, for some unitaries $U,V\in\bM_d$,
$$
|A+B|_{\qsym} \le U|A|_{\qsym}U^* +V|B|_{\qsym}V^*.
$$
\end{theorem}

\vskip 5pt

This   is  the symmetrized companion of the celebrated  Thompson's inequality \cite{T}.  Theorem \ref{thZ}
    was motivated by  the following  special case   \cite{BLsym}:

\vskip 5pt
\begin{cor}\label{symq} Let $Z\in\bM_d$. Then, for some unitaries $U,V\in\bM_d$,
$$
 |Z|_{\qsym} \le U |{\mathrm{Re\,}}Z| U^* + V |{\mathrm{Im\,}}Z|V^*.
$$
\end{cor}

\vskip 5pt
Here \eqref{compsqs} shows that the left-hand side may be replaced by $|Z|_{\sym}$. 
Despite  this observation Theorem \ref{thZ} does not hold for  the symmetric modulus $|\cdot|_{\sym}$.   Two matrices $A,B\in\bM_3$ are given in \cite{Zsharp} with $$\lambda_1^{\downarrow}(|A+B|_{\sym})=\sqrt{2} \left\{ \lambda_1^{\downarrow}(|A|_{\sym})+\lambda_1^{\downarrow}( |B|_{\sym})\right\}.$$ Here, in dimension $d\ge 3$,  the constant $\sqrt{2}$ is optimal. Indeed,   a result of \cite{BLtriang} states a submajorization form of the triangle inequality as follows.

\vskip 5pt
\begin{theorem}\label{BLsym}   Let $\{X_k\}_{k=1}^m$ be in $\bM_d$. Then, 
$$
\left|\sum_{k=1}^m X_k\right|_{\sym} \prec_w \sqrt{2} \sum_{k=1}^m \left|X_k\right|_{\sym}.
$$
\end{theorem}

In his impressive work \cite{Zsharp},  Zhang got the idea to study triangle inequalities between different types of moduli. In particular, he observed the following hybrid version of Theorem \ref{BLsym}:

\vskip 5pt
\begin{theorem}\label{Zhyb}   Let $\{X_k\}_{k=1}^m$ be in $\bM_d$. Then, 
$$
\left|\sum_{k=1}^m X_k\right| \prec_w \sqrt{2} \sum_{k=1}^m \left|X_k\right|_{\qsym}.
$$ 
\end{theorem}

Here both the usual  modulus and the quadratic symmetric modulus appear simultaneously. The constant $\sqrt{2}$ is optimal in any dimension $d\ge 2$.  Theorem \ref{Zhyb} is remarkable if one keeps in mind the well known submajorization \cite[Proposition 2.1]{L}:

\vskip 5pt
\begin{prop}\label{Leeineq}   Let $\{X_k\}_{k=1}^m$ be in $\bM_d$. Then, 
$$
\left|\sum_{k=1}^m X_k\right| \prec_w \sqrt{m} \sum_{k=1}^m \left|X_k\right|
$$ 
\end{prop}

Hence replacing the usual modulus in  the right-hand side  by the quadratic symmetric modulus yields    the pleasant constant $\sqrt{2}$, regardless of the dimension. The constant $\sqrt{m}$ is optimal whenever $m\le d$.  If $m> d$, then the multiplicative constant  $\sqrt{m}$ can be replaced by the optimal one $\sqrt{d}$, see \cite[Theorem 1.7]{Zsharp}.

The purpose of Section 2 is to refine Theorem \ref{Zhyb} at the operator level. Our main tool is a polar decomposition associated with the quadratic symmetric modulus. It yields unitary-orbit inequalities strengthening Zhang's submajorization estimate and leads to several eigenvalue consequences. We also establish, in Section 3,  a polar decomposition for the maximal modulus associated with Olson's spectral order.   This allows us to obtain estimates for the maximal modulus similar to those of Section 2 for the quadratic modulus.

\section{Unitary orbits and  log-majorization }

The following lemma may be regarded as a polar decomposition for the quadratic symmetric modulus.

\vskip 5pt
\begin{lemma}\label{polarq} Let $A\in\bM_d$. Then there exists a contraction $C\in\bM_d$ such that
$$
 A =\sqrt{2}|A|_{\qsym}^{1/2} C |A|_{\qsym}^{1/2}.
$$
\end{lemma}

\vskip 5pt
\begin{proof}
By a limit argument we may suppose invertibility of $A$. It then suffices to show that
$$
C=(|A|^2+|A^*|^2)^{-1/4}A(|A|^2+|A^*|^2)^{-1/4}
$$
is a contraction. Note that the polar decomposition $A=|A^*|^{1/2} U|A|^{1/2}$ gives
$$
C=(|A|^2+|A^*|^2)^{-1/4}\left(|A^*|^{1/2} U|A|^{1/2}\right)(|A|^2+|A^*|^2)^{-1/4}.
$$
Thus $C=XY$ where
$$
X=(|A|^2+|A^*|^2)^{-1/4}|A^*|^{1/2}, \qquad Y= U|A|^{1/2}(|A|^2+|A^*|^2)^{-1/4}.
$$
Hence $\|C\|_{\infty}^2\le \|X\|_{\infty}^2\|Y\|_{\infty}^2=\| X^*X\|_{\infty}\|YY^*\|_{\infty}$, that is
$$
\| C\|_{\infty}^2 \le  \| |A^*|^{1/2}(|A|^2+|A^*|^2)^{-1/2}|A^*|^{1/2}\|_{\infty} \| |A|^{1/2}(|A|^2+|A^*|^2)^{-1/2}|A|^{1/2}\|_{\infty}.
$$
Now, since $t\mapsto t^{-1/2}$ is operator decreasing, we have
$$
|A^*|^{1/2}(|A|^2+|A^*|^2)^{-1/2}|A^*|^{1/2} \le |A^*|^{1/2}(|A^*|^2)^{-1/2}|A^*|^{1/2} =I,
$$
and 
$$
|A|^{1/2}(|A|^2+|A^*|^2)^{-1/2}|A|^{1/2} \le |A|^{1/2}(|A|^2)^{-1/2}|A|^{1/2} =I.
$$
The last three inequalities together give
$
\| C\|_{\infty}^2 \le 1.
$
\end{proof}

\vskip 5pt
Now we give a hybrid comparison between the usual modulus and the quadratic symmetric modulus with the help of the geometric mean $\#$ on $\bM_d^+$. The result is strongly influenced by Zhang's  Theorem \ref{Zhyb}.

\vskip 5pt
\begin{theorem} \label{main} Let $\{A_k\}_{k=1}^m$ be in $\bM_d$. Then, 
$$
\left| \sum_{k=1}^m A_k\right| \le \sqrt{2} \left\{ \sum_{k=1}^m |A_k|_{\qsym}\right\} \#V\left\{
 \sum_{k=1}^m |A_k|_{\qsym}\right\}V^*
$$
for some unitary matrix $V\in\bM_d$.
\end{theorem}

\vskip 5pt
\begin{proof} By Lemma \ref{polarq}, for every $A\in\bM_d$, the block matrix
$$
\begin{bmatrix} \sqrt{|A|^2+|A^*|^2} & A \\
A^*& \sqrt{|A|^2+|A^*|^2} 
\end{bmatrix}
$$
is positive semidefinite. Hence so is the sum of $m$ such matrices,
$$
\begin{bmatrix} \sum_k \sqrt{|A_k|^2+|A_k^*|^2} & \sum_k A_k \\
 \sum_k  A_k^*& \sum_k  \sqrt{|A_k|^2+|A_k^*|^2} 
\end{bmatrix}
$$
where each sum runs over $k\in\{1,\ldots, m\}$. Now consider the polar decomposition
$
 \sum_k A_k = U| \sum_k A_k |
$
and perform the unitary congruence
$$ 
\begin{bmatrix} U^*&0 \\ 0&I\end{bmatrix}
\begin{bmatrix} \sum_k \sqrt{|A_k|^2+|A_k^*|^2} & \sum_k A_k \\
 \sum_k  A_k^*& \sum_k  \sqrt{|A_k|^2+|A_k^*|^2} 
\end{bmatrix}
\begin{bmatrix} U&0 \\ 0&I\end{bmatrix}
$$
to get the positive semidefinite block-matrix:
$$
\begin{bmatrix} U^*\left\{\sum_k \sqrt{|A_k|^2+|A_k^*|^2}\right\}U& \left|\sum_k A_k\right|\\
\left| \sum_k  A_k\right|& \sum_k  \sqrt{|A_k|^2+|A_k^*|^2} 
\end{bmatrix}
$$
By the maximal characterization of the geometric mean,
$$
\left| \sum_k  A_k\right| \le 
\left(U^*\left\{\sum_k \sqrt{|A_k|^2+|A_k^*|^2}\right\}U\right)\#\left\{\sum_k \sqrt{|A_k|^2+|A_k^*|^2}\right\}
$$
Using $X\#Y=Y\#X$ and setting $V=U^*$ completes the proof.
\end{proof}

\vskip 5pt
From the matrix AGM inequality we obtain:

\vskip 5pt
\begin{cor}\label{cormean}  Let $\{A_k\}_{k=1}^m$ be in $\bM_d$. Then, 
$$
\left| \sum_{k=1}^m A_k\right| \le \frac{\sqrt{2}}{2}  \sum_{k=1}^m \left(|A_k|_{\qsym}+V|A_k|_{\qsym}V^*\right)
$$
for some unitary matrix $V\in\bM_d$.
\end{cor}

\vskip 5pt
Corollary  \ref{cormean} clearly implies Zhang's Theorem \ref{Zhyb}. Besides the series of eigenvalue inequalities of the form \eqref{submajdef}, Corollary  \ref{cormean} adds some new ones, by application of the classical Weyl inequalities:

\vskip 5pt
\begin{cor}\label{cormean2} Let $\{A_k\}_{k=1}^m$ be in $\bM_d$. Then, for all $i,j=0,1,\ldots$,
$$
\lambda_{1+i+j}^{\downarrow } \left(\left| \sum_{k=1}^m A_k\right| \right) \le \frac{ \lambda_{1+i}^{\downarrow}\left( \sum_{k=1}^m |A_k|_{\qsym}\right) +
\lambda_{1+j}^{\downarrow}\left( \sum_{k=1}^m |A_k|_{\qsym}\right)}{\sqrt{2}}.
$$
\end{cor}

\vskip 5pt
Another consequence of Theorem \ref{main} is  a weak log-majorization $\prec_{\mathrm{wlog}}$ on $\bM_d^+$ given in the next corollary. Recall that, for $A,B\in\bM_d^+$, we write $A\prec_{\mathrm{wlog}}B$ to mean that the following $d$ inequalities hold:
$$
\prod_{j=1}^p\lambda_{j}^{\downarrow} (A) \le \prod_{j=1}^p\lambda_{j}^{\downarrow} (B), \qquad p=1,\ldots,d.
$$
It is well known that $A\prec_{\mathrm{wlog}} B\Rightarrow A\prec_wB$.

\vskip 5pt
\begin{cor}  Let $\{A_k\}_{k=1}^m$ be in $\bM_d$. Then, 
$$
\left| \sum_{k=1}^m A_k\right| \prec_{\mathrm{wlog}} \sqrt{2} \sum_{k=1}^m |A_k|_{\qsym}.
$$
\end{cor}

\vskip 5pt
\begin{proof} We have
$$
 \left\{ \sum_{k=1}^m |A_k|_{\qsym}\right\} \#V\left\{
 \sum_{k=1}^m |A_k|_{\qsym}\right\}V^*=  \left\{ \sum_{k=1}^m |A_k|_{\qsym}\right\}^{1/2}\!\! W \!\cdot\! V\left\{ \sum_{k=1}^m |A_k|_{\qsym}\right\}^{1/2}\!\!\!V^*
$$
for some unitary $W\in\bM_d$. Using Theorem \ref{main}, we infer
\begin{equation}\label{W}
\left| \sum_{k=1}^m A_k\right| \le   \sqrt{2} \left\{ \sum_{k=1}^m |A_k|_{\qsym}\right\}^{1/2} \!\!W \!\cdot\!  V\left\{ \sum_{k=1}^m |A_k|_{\qsym}\right\}^{1/2}\!\!\!V^*.
\end{equation}
This combined with Horn's multiplicative inequalities for singular values yields 
$$
\prod_{j=1}^p\lambda_{j}^{\downarrow} \left(\left| \sum_{k=1}^m A_k\right| \right)
\le
\prod_{j=1}^p\lambda_{j}^{\downarrow} \left(\sqrt{2}\sum_{k=1}^m |A_k|_{\qsym}\right)
$$
for all $p=1,\ldots, d$. This is exactly our weak log-majorization relation $\prec_{\mathrm{wlog}}$. 
\end{proof}

\vskip 5pt
From another classical singular value inequality of Weyl, we obtain from \eqref{W}:

\vskip 5pt
\begin{cor}\label{corprod}  Let $\{A_k\}_{k=1}^m$ be in $\bM_d$. Then, for all $i,j=0,1,\ldots$,
$$
\lambda_{1+i+j}^{\downarrow 2} \left(\left| \sum_{k=1}^m A_k\right| \right) \le \, 2\, \lambda_{1+i}^{\downarrow}\left( \sum_{k=1}^m |A_k|_{\qsym}\right)
\lambda_{1+j}^{\downarrow}\left( \sum_{k=1}^m |A_k|_{\qsym}\right).
$$
\end{cor}

Note that  Corollary \ref{cormean2} follows from Corollary \ref{corprod} by applying the AGM inequality. A nice special case from either Corollary \ref{cormean} or Corollary \ref{corprod} is:

\vskip 5pt
\begin{cor}  Let $\{A_k\}_{k=1}^m$ be in $\bM_d$. Then, for all $j=0,1,\ldots$,
$$
\lambda_{1+2j}^{\downarrow } \left(\left| \sum_{k=1}^m A_k\right| \right) \le \sqrt{2}\, \lambda_{1+j}^{\downarrow}\left( \sum_{k=1}^m |A_k|_{\qsym}\right).
$$
\end{cor}

\vskip 5pt
Next we turn to averages over a unitary orbit. It has  recently been shown \cite{BLaverage} that any  weak majorization relation $A\prec_w B$ in $\bM_d^+$ is equivalent to an inequality with an average over a unitary orbit:
 $$
A \le \frac{1}{d}\sum_{k=1}^d U_kBU_k^*
$$
for some unitary matrices $U_k\in\bM_d$. Therefore, Theorem \ref{Zhyb} is equivalent to:

\vskip 5pt
\begin{theorem}\label{Zhyb2}   Let $\{X_k\}_{k=1}^m$ be in $\bM_d$. Then, there exist unitary matrices  $\{V_i\}_{i=1}^d$  in $\bM_d$ such that
$$
\left|\sum_{k=1}^m X_k\right| \le \frac{\sqrt{2}}{d}\sum_{i=1}^d V_i\left\{ \sum_{k=1}^m \left|X_k\right|_{\qsym}\right\}V_i^*.
$$ 
\end{theorem}

\vskip 5pt
That Theorem \ref{Zhyb2} implies Theorem \ref{Zhyb} is immediate. The converse follows from the orbit-average
characterization quoted above. Going back to Corollary
\ref{cormean}, we see that we can actually choose an average with only two elements in the unitary orbits of $\sqrt{2}\sum_{k=1}^m \left|X_k\right|_{\qsym}$, 
$$
\left|\sum_{k=1}^m X_k\right| \le \frac{\sqrt{2}}{2}\left(\sum_{k=1}^m \left|X_k\right|_{\qsym} + V\left\{ \sum_{k=1}^m \left|X_k\right|_{\qsym}\right\} V^*\right)
$$ 
for some unitary matrix $V\in\bM_d$.

\section{Related results for the maximal modulus}

We have focused on a hybrid inequality involving the usual modulus and its quadratic symmetrized version. Hybrid inequalities with the symmetric modulus and the quadratic one are also interesting; for instance \cite[Corollary 2.12]{BLaverage} 
states:

\vskip 5pt
\begin{theorem} Let $Z\in \bM_d$. Then there exist unitary matrices $\{V_i\}_{i=1}^d$  in $\bM_d$ such that 
$$
|Z|_{\qsym} \le \frac{ \sqrt{2}}{d}  \sum_{i=1}^d V_i |Z|_{\sym}V_i^*.
$$
The constant $\sqrt{2}$ cannot be diminished.
\end{theorem}

A simpler relation between the symmetric modulus and the quadratic symmetric modulus is established in the next proposition. It improves \eqref{compsqs}.
 By operator convexity of $t^2$,
$$
\left( \frac{|Z|+|Z^*|}{2}\right)^2\le \frac{|Z|^2+|Z^*|^2}{2}
$$
equivalently
 \begin{equation}\label{compsqs2}
|Z|^2_{\sym}\le |Z|^2_{\qsym}.
\end{equation} 
This inequality implies  \eqref{compsqs} as $\sqrt{t}$ is operator monotone. The following proposition   highlights the geometric significance of  \eqref{compsqs2}. 

\vskip 5pt
\begin{prop}\label{ub1}
Let  $Z\in \bM_d$ and let ${\mathcal{U}}$ be the  unit ball of $\bC^d$. Then,
$$
|Z|_{\sym}({\mathcal{U}})\subset |Z|_{\qsym}({\mathcal{U}}).
$$
\end{prop}

\vskip 5pt
\begin{proof} Thanks to \eqref{compsqs2} it suffices to show that, for $A,B\in\bM_d^+$, we have the equivalence
$$
A^2\le B^2 \iff A({\mathcal{U}}) \subset B({\mathcal{U}}).
$$
Since each side ensures that the range of $A$ is included in the range of $B$, restricting our operators to the range of $B$ if necessary, we may assume that $B$ is invertible. Then
$$
A^2\le B^2 \iff B^{-1}A^2B^{-1} \le I \iff \| B^{-1}A\|_{\infty} \le 1 \iff B^{-1}A({\mathcal{U}}) \subset{\mathcal{U}}, 
$$equivalently, $A({\mathcal{U}}) \subset B({\mathcal{U}})$.
\end{proof}

The symmetric modulus $|Z|_{\sym}$ is the $\ell^1$-mean of  $|Z^*|$ and $|Z|$, and 
the quadratic symmetric modulus $|Z|_{\qsym}$ is their $\ell^2$-mean. 

There exists another interesting symmetrized modulus, which corresponds to a $\ell^{\infty}$-mean.  This is the maximal modulus
 $| Z|_{\vee}$ which corresponds to the maximum of $|Z|$ and $|Z^*|$ for a natural order, Olson's order \cite{Ols}.  This order endows the positive  matrices with a lattice structure.  We may define the maximal modulus of $Z\in\bM_d$ with Kato's formula \cite{K},
$$
|Z|_{\vee}:=\lim_{n\to \infty} (|Z|^n +|Z^*|^n)^{1/n}.
$$
A concise treatment of the Olson order, and the proof of this  formula is given in \cite{BLaverage}.
We think that these three symmetric moduli   are especially  important ones.
Our next result may be regarded as a polar decomposition for the maximal modulus.

\begin{lemma}\label{polarmax}
Let  $A\in \bM_d$. Then, for some contraction $K\in\bM_d$,
$$
A=|A|_{\vee}^{1/2} K |A|_{\vee}^{1/2}.
$$
\end{lemma}

\vskip 5pt
\begin{proof} By restricting to the range of $|A|+|A^*|$ if necessary, we may assume that $|A|^n +|A^*|^n$ is invertible for each integer $n>0$. It then suffices to show that
$$
C_n=(|A|^n+|A^*|^n)^{-1/2n}A(|A|^n+|A^*|^n)^{-1/2n}
$$
is a contraction. Indeed $(|A|^n+|A^*|^n)^{-1/2n}\to |A|_{\vee}^{-1/2}$ as $n\to\infty$, and we may extract from $\{C_n\}$ a subsequence converging to a contraction $K$. The proof is then similar to that  of Lemma \ref{polarq}. The polar decomposition $A=|A^*|^{1/2} U|A|^{1/2}$ gives
$$
C_n=(|A|^n+|A^*|^n)^{-1/2n}\left(|A^*|^{1/2} U|A|^{1/2}\right)(|A|^n+|A^*|^n)^{-1/2n}.
$$
Thus, by the Cauchy-Schwarz inequality $\| XY\|_{\infty}\le \| X^*X\|_{\infty}^{1/2} \| YY^*\|_{\infty}^{1/2}$ applied to
$$
X=(|A|^n+|A^*|^n)^{-1/2n}|A^*|^{1/2}, \qquad Y= U|A|^{1/2}(|A|^n+|A^*|^n)^{-1/2n},
$$
we obtain
$$
\| C_n\|_{\infty}^2 \le  \| |A^*|^{1/2}(|A|^n+|A^*|^n)^{-1/n}|A^*|^{1/2}\|_{\infty} \| |A|^{1/2}(|A|^n+|A^*|^n)^{-1/n}|A|^{1/2}\|_{\infty}.
$$
Now, since $t\mapsto t^{-1/n}$ is operator decreasing, we have
$$
|A^*|^{1/2}(|A|^n+|A^*|^n)^{-1/n}|A^*|^{1/2} \le |A^*|^{1/2}(|A^*|^n)^{-1/n}|A^*|^{1/2} = I,
$$
and 
$$
|A|^{1/2}(|A|^n+|A^*|^n)^{-1/n}|A|^{1/2} \le |A|^{1/2}(|A|^n)^{-1/n}|A|^{1/2} =I.
$$
The last three inequalities together give
$
\| C_n\|_{\infty}^2 \le 1.
$
\end{proof}

Having at our disposal Lemma \ref{polarmax} we see that the block matrix
$$
\begin{bmatrix} |A|_{\vee}& A \\
A^*& |A|_{\vee} 
\end{bmatrix}
$$
is positive semidefinite. Arguing as in the quadratic symmetric case we thus obtain the following results:

\vskip 5pt
\begin{theorem}  Let $\{A_k\}_{k=1}^m$ be in $\bM_d$. Then, 
$$
\left| \sum_{k=1}^m A_k\right| \le  \left\{ \sum_{k=1}^m |A_k|_{\vee}\right\} \#V\left\{
 \sum_{k=1}^m |A_k|_{\vee}\right\}V^*
$$
for some unitary matrix $V\in\bM_d$.
\end{theorem}

\vskip 5pt
\begin{cor}\label{cormeanmax}  Let $\{A_k\}_{k=1}^m$ be in $\bM_d$. Then, 
$$
\left| \sum_{k=1}^m A_k\right| \le \frac{  \sum_{k=1}^m \left(|A_k|_{\vee}+V|A_k|_{\vee}V^*\right)}{2}
$$
for some unitary matrix $V\in\bM_d$.
\end{cor}

\vskip 5pt
\begin{cor}  Let $\{A_k\}_{k=1}^m$ be in $\bM_d$. Then, 
$$
\left| \sum_{k=1}^m A_k\right| \prec_{\mathrm{wlog}}  \sum_{k=1}^m |A_k|_{\vee}.
$$
\end{cor}

\section{Concluding remarks}

This note is a small contribution to the study of symmetric moduli. We believe that this line of research will shed further insight on matrix inequalities and decompositions. In fact, there exists already a series of impressive results, in particular in Zhang's papers \cite{Ztriang}, \cite{Zsharp}. We may also mention our contributions \cite{BLtriang}, \cite{BLaverage}, in which several results and questions are proposed.

These symmetric moduli are natural  representatives in $\bM_d^+$ for matrices in $\bM_d$. Results such as Zhang's Theorem \ref{thZ} or the maximal polar decomposition Lemma \ref{polarmax} invite  further investigation.

\vskip 15pt
\noindent
Jean-Christophe Bourin

\noindent
Université Marie et Louis Pasteur, CNRS, LmB (UMR 6623), F-25000 Besançon, France.

\noindent

\noindent

\noindent
Email: jcbourin@univ-fcomte.fr

  \vskip 15pt\noindent
Eun-Young Lee

\noindent
Department of Mathematics, KNU-Center for Nonlinear Dynamics,

\noindent
Kyungpook National University,

\noindent
Daegu 702-701, Korea.

\noindent
 Email: eylee89@knu.ac.kr


\begin{thebibliography}{99}
{\small

 


\bibitem{BLsym} J.-C.\ Bourin and E.-Y.\ Lee, Matrix parallelogram laws and symmetric moduli, {\it Internat.\ J.\ Math.} 37
(2026), no.\ 3, 2650018.

\bibitem{BLtriang} J.-C.\ Bourin and E.-Y.\ Lee, Triangle inequalities for the operator symmetric modulus, {\it Proc. Amer. Math. Soc.}, in press, arXiv:2602.19607.


\bibitem{BLaverage} J.-C.\ Bourin and E.-Y.\ Lee, Averages over matrix unitary orbits  and  spectral order, preprint, arXiv:2606.15624, 2026.

\bibitem{HiP} F.\ Hiai, D.\ Petz, Introduction to matrix analysis and applications. Universitext. Springer, Cham; Hindustan Book Agency, New Delhi, 2014.



\bibitem{K} 
T. Kato, Spectral order and a matrix limit theorem, {\it Linear  Multilinear Algebra} 8 (1979),
15--19.

\bibitem{L} E.-Y. Lee, How to compare the absolute values of operator sums and the sums of absolute values?  {\it Oper.\ Matrices} 6 (2012), no. 3, 613--619.

\bibitem{Ols} M.\ P.\ Olson, The selfadjoint operators of a von Neumann algebra form a conditionally complete
lattice, {\it Proc. Amer. Math. Soc.} 28 (1971), 537--544.




\bibitem{T}   R.\ C.\ Thompson, Convex and concave functions of singular
values of matrix sums, {\it Pacific J. Math.}\  66 (1976), 285--290.





\bibitem{Ztriang} T.\ Zhang, An operator triangle inequality for the quadratic symmetric modulus, preprint, arXiv:2602.01463, 2026.

\bibitem{Zsharp} T.\ Zhang, Operator symmetric moduli and sharp triangle inequalities, preprint,  arXiv:2603.01046, 2026.
}

\end{thebibliography}
\end{document}